\documentclass[12pt,twoside,a4paper]{amsart}

\usepackage[margin=2.3cm]{geometry}

\usepackage{amssymb,mathtools}
\usepackage{xcolor}
\usepackage{hyperref}

\newtheorem{thm}{Theorem}[section]
\newtheorem{remark}{Remark}
\newtheorem{lem}[thm]{Lemma}
\newtheorem{prop}[thm]{Proposition}
\newtheorem{cor}[thm]{Corollary}
\theoremstyle{definition}
\newtheorem{defn}[thm]{Definition}
\theoremstyle{remark}

\numberwithin{equation}{section}
\allowdisplaybreaks

\newcommand{\R}{\mathbb{R}}
\newcommand{\Z}{\mathbb{Z}}
\newcommand{\N}{\mathbb{N}}

\newcommand{\HD}{\dim_H}
\newcommand{\Hau}{\mathcal{H}}

\newcommand{\supp}{\operatorname{supp}}

\begin{document}


\title[A curved three-point pattern problem]{A curved three-point pattern problem for fractal sets on the real line}

\author{Surjeet Singh Choudhary} 
\address{National Center for Theoretical Sciences, National Taiwan University, Taipei 106, Taiwan.}
\email{surjeet19@ncts.ntu.edu.tw}

\author{Chong-Wei Liang}
\address{Department of Mathematics, National Taiwan University, Taiwan.}
\email{d10221001@ntu.edu.tw}

\author{Chun-Yen Shen} 
\address{Department of Mathematics, National Taiwan University, Taiwan.}
\email{cyshen@math.ntu.edu.tw}

\date{}

\begin{abstract}
We study the occurrence of curved three-point configurations in fractal subsets of the real line. We prove that if \(E \subset [0,1]\) is a compact set with sufficiently large Hausdorff dimension, then \(E\) contains a curved three-point progression associated with a broad class of nonlinear functions.

Our approach can also show the existence of the curved three-point pattern under the assumption that the Hausdorff content of \(E\) is bounded away from zero. The class of functions includes, in addition to polynomials with vanishing constant term, nonlinear functions such as
\[
t^k \log(1+t), \quad \forall k \geq 1.
\]
\end{abstract}

\maketitle
\section{Introduction}
The study of arithmetic patterns in sets of integers has been a central theme in additive combinatorics. A landmark result in this direction is Roth's theorem \cite{MR0051853}, which asserts that any subset of the integers with positive density contains a nontrivial three-term arithmetic progression. This was later generalized by Szemerédi \cite{MR0369312}, who proved that any subset of the integers with positive density contains arbitrarily long arithmetic progressions. Moreover, the celebrated theorem of Green and Tao \cite{MR2415379} showed that the primes themselves contain arbitrarily long arithmetic progressions, revealing a deep and unexpected structure within a sparse set.

These results have inspired a parallel line of inquiry in geometric measure theory and harmonic analysis: to what extent do analogous pattern-forming phenomena persist in subsets of Euclidean space with large Hausdorff dimension? In this continuous setting, density is replaced by dimensional considerations, and arithmetic structure is studied through analytic and geometric methods.

A natural starting point is to consider linear configurations of the form
\[
\{x, x-t, x-2t\}.
\]
In contrast with the discrete setting, it is known that such patterns can fail dramatically: there exist compact sets $E \subset [0,1]$ with full Hausdorff dimension $\dim_H(E)=1$ that do not contain any nontrivial three-term arithmetic progression \cite{MR1704757}. This striking phenomenon highlights a fundamental obstruction in the continuous setting and shows that large Hausdorff dimension alone is insufficient to guarantee the existence of linear patterns. For more results in this direction, refer the papers \cite{MR3016405} of Maga, \cite{MR3672917} of Máthé, \cite{MR4238597} of
Denson, Łaba, and Zahl, and \cite{MR4305962} of Yavicoli. 

This raises a natural question: {\it Can one recover Roth-type phenomena by moving beyond linear configurations?}

Many nonlinear configuration existence theorems require strong Fourier analytic hypothesis, such as the lower bound for the Fourier dimension, or the existence of a measure that obeys a ball condition and a Fourier condition \cite{MR3481177,MR4418720,MR2545245,MR3531369}. It was unknown that whether these Fourier analytic hypothesis is necessary. The first breakthrough is attained by Kuca, Orponen and Sahlsten in \cite{MR4609784}, who established a continuous S\'ark\"ozy type theorem for a parabola. Indeed, they construct a measure with energy and spectral gap condition, which avoids the assumption on Fourier analytic condition. Later, Bruce and Pramanik in \cite{MR4966567} employ their construction to prove the existence of a two-point pattern determined by curves. Apart from the two-point pattern result on higher dimension $d\geq2$, Zhu in \cite{MR5038694} utilizes this idea as well as the established smoothing inequality in \cite{MR4418720} to show the existence of the three-point pattern in real line for the quadratic polynomials. 

 Various works have established the existence of nonlinear patterns in sets of large Hausdorff dimension, particularly for polynomial-type configurations \cite{MR0942826,GDH,MR4831142,MR4410765,XSMW}. These results suggest that nonlinear structures may exhibit greater rigidity than their linear counterparts in the fractal setting. A large portion of these arguments depend heavily on the smoothing inequality. To be more specific, the Fourier method approach suggests that the non-zero lower bound for the configuration integral guarantees the existence of the pattern. The smoothing inequality plays a crucial role when estimating the configuration integral. 

In this paper, we investigate nonlinear three-point configurations of the form
\begin{align*}
\left\{x,\, x-t,\, x-\gamma(t)\right\},
\end{align*}
where $\gamma$ belongs to a certain class of nonlinear functions. Our main result shows that, under a natural non-degeneracy condition on $\gamma$, such configurations must occur in any set of sufficiently large Hausdorff dimension. 

Before introducing our family of nonlinear functions, we need some notations.\\ Let $\Theta$ be a family of real-analytic curves $\gamma:[0,1]\to \R$ which is not an affine map with $\gamma(0)=0$ satisfies that\\
    $(1)$\,There do not exist $a,b_1,b_2\in\mathbb{C}$ with $a,b_1\ne 0$ such that $\gamma$ is of the form $\frac{1}{a}\log(b_1e^{at}+b_2)$,\\
$(2)$\,
For all $t \in (0,1)$, $\gamma'(t)\neq0$ and $|1-\gamma'(t)|+|\gamma''(t)|\neq0$,\\
$(3)$\,$|\gamma(t)|\lesssim|t|$ for all sufficiently small $t$.
\smallskip

Here comes our family of curves.
\begin{defn}
Let $M\geq1$ be an integer. Define $\Theta(M)$ to be the family of curves contained in $\Theta$ such that $|\gamma(t)|\leq M|t|$ for all sufficiently small $t$.
\end{defn}

We give some examples that are not polynomial and are contained in our family of curves. Let $I=[0,1]$, then one can take $\gamma$ to be the following curves:
\begin{align*}
    t^2\log(1+t),\,t-\sin(t),\,\tan(t)-t,\,t-\arctan(t),\,\arcsin(t)-t
\end{align*}
or any linear combinations of these curves; besides, all the functions of the form $t^k\log(1+t)$ for $k\geq1$ are in the class $\Theta(1)$. \\

Given a function $\gamma(t)$ and $\lambda>0$, throughout this paper, we denote $\gamma_{\lambda}(t): = \lambda^{-1}\gamma(\lambda t)$.

\begin{thm}\label{mainthm}
Let $\gamma \in \Theta(1)$. Then there exists $\varepsilon = \varepsilon_\gamma \in (0,1)$ such that any compact set $E \subset [0,1]$ with
\[
\dim_H(E) > 1 - \varepsilon
\]
contains a nontrivial configuration of the form
\[
\{x,\ x - t,\ x - \gamma_\lambda(t)\}
\]
for some $x \in \mathbb{R}$ and $\lambda, t > 0$.
\end{thm}

This result can be viewed as a nonlinear analogue of Roth-type theorems in the fractal setting. In particular, it demonstrates that while linear configurations may be avoidable even at full dimension, a broad class of nonlinear patterns cannot be avoided once the Hausdorff dimension is sufficiently close to one.

 \begin{remark}
     As observed in \cite{BK,MR5038694}, one can show Theorem \ref{mainthm} under the assumption that the Hausdorff content of $E$ has a positive lower bound. We record such a result here:

     Let $\gamma \in \Theta(1)$. Then there exists $\sigma=\sigma_\gamma\in (0,1)$ and $\delta\in(0,1)$ such that for any compact set $E \subset [0,1]$ with $\Hau^{\sigma}_\infty(E)\geq (1-\delta)$, the set $E$
contains a nontrivial configuration of the form
\[
\{x,\ x - t,\ x - \gamma(t)\}
\]
for some $x \in \mathbb{R}$ and $t > 0$.
 \end{remark}

\subsubsection{Overview of the method.}
Our approach combines techniques from harmonic analysis and geometric measure theory. The key ingredients are:

\begin{itemize}
\item The \emph{smoothing inequality} in \cite{MR4776384} for a nonlinear averaging operator associated with the curve $\gamma$, which captures a form of analytic regularity for the configuration operator.

\item The construction of a probability measure supported on $E$ that satisfies both an \emph{energy condition} and a \emph{spectral gap condition} \cite{MR4966567,MR4609784}. This measure plays the role of a pseudorandom object, allowing us to control different frequency components.

\item A decomposition of the configuration integral into main and error terms via a high--low frequency analysis. The main term is shown to be positive, while the error terms are controlled using the smoothing inequality and the spectral gap property.
\end{itemize}

These ingredients together yield a robust mechanism for detecting nonlinear configurations in fractal sets of large Hausdorff dimension.

\medskip

\noindent
\textbf{Organization of the paper.}
In Section \ref{sec:2}, we review the necessary background from geometric measure theory, including Frostman measures and Hausdorff content in a metric space setting. In Section \ref{sec:3}, we establish the smoothing inequality associated with our family of curves and develop a general mechanism for detecting three-point configurations. In Section \ref{sec:4}, we construct measures satisfying both the energy and spectral gap conditions. Finally, in Section \ref{sec:5}, we combine these ingredients to prove Theorem \ref{mainthm}.

\section{Preliminaries}\label{sec:2}
\subsection{Frostman measure and the Riesz energy}

This subsection introduces the necessary background from geometric measure theory. 
Most of the results presented here can be found in \cite[Chapters 2 and 3]{MR3617376}.

We begin with the notion of a Frostman measure.

\begin{defn}[Frostman measure]~\\
Let $s>0$. A measure $\mu$ is called an $s$-Frostman measure if
\begin{equation} \label{definition of Frostman measure}
\mu(B(x,r)) \lesssim r^s,
\qquad 
\forall\, x \in \R^d,\ \forall\, r>0.
\end{equation}
\end{defn}

We do not assume that an $s$-Frostman measure $\mu$ is a probability measure. 
However, if $\mu$ has compact support, then it is finite. 
The following theorem, known as Frostman's lemma, shows that fractal sets can be represented by Frostman measures.

\begin{thm}[Frostman's lemma] \label{Frostman's lemma}~\\
Let $s \in (0,d]$. For a Borel set $E \subset \R^d$, one has $\Hau^{s}(E)>0$ if and only if there exists $\mu \in \mathbf{M}(E)$ such that $\mu$ is an $s$-Frostman measure, where $\Hau^{s}$ denotes the $s$-dimensional Hausdorff measure, and $\mathbf{M}(E)$ denotes the collection of finite Borel measures with compact support contained in $E$.
\end{thm}

Next, we introduce the notion of the Riesz energy.

\begin{defn}[Riesz energy]~\\
Let $s>0$. The Riesz energy of a Borel measure $\mu$, denoted by $I_s(\mu)$, is defined by
\[
I_s(\mu)
:=
\int_{\R^d} \int_{\R^d} |x-y|^{-s} \, d\mu(y)\, d\mu(x)
=
\int_{\R^d} (k_s * \mu)(x)\, d\mu(x),
\]
where $k_s(x)=|x|^{-s}$ for all nonzero $x \in \R^d$.
\end{defn}

The study of Frostman measures is often associated with the Riesz energy via the following proposition.

\begin{prop}
Let $\mu \in \mathbf{M}(\R^d)$ be an $s$-Frostman measure. Then
\[
I_t(\mu)\lesssim_{s,t,d} \mu(\R^d),
\qquad \text{for all } t\in[0,s).
\]
\end{prop}

Finally, we relate the Riesz energy to Sobolev norms.

\begin{prop}\label{EnergyandSobolev}
Let $\mu \in \mathbf{M}(\R^d)$ and $s\in (0,d)$. Then
\[
I_s(\mu)=\gamma(d,s)
\int_{\R^d} |\widehat{\mu}(\xi)|^2 |\xi|^{s-d} \, d\xi
\simeq \gamma(d,s)\,\|\mu\|^2_{H^{(s-d)/2}(\R^d)},
\]
for some constant $\gamma(d,s)$, where
\begin{equation}\label{Sobolevdef}
\|f\|^2_{H^{\sigma}(\R^d)}
:=
\int_{\R^d} |\widehat{f}(\xi)|^2 (1+|\xi|^2)^{\sigma} \, d\xi.
\end{equation}
\end{prop}

\subsection{Generalized Hausdorff measure}

In this subsection, we recall the general Hausdorff measure in a metric space; see \cite[Chapter 4]{MR1333890}.  

\subsection{General setup}
Let $(\mathbb{X},|\cdot|)$ be a metric space, let $\mathcal{F}$ be a family of subsets of $\mathbb{X}$, and let $h:[0,\infty)\to [0,\infty)$ be a non-decreasing function with $h(0)=0$. 
For any subset $E \subset \mathbb{X}$, let $d(E)$ denote its {\it diameter}. 
We make the following two assumptions.

\begin{enumerate}
\item For every $\delta \geq 0$, there exists a family $\{E_i\}_{i\in\mathbb{N}} \subset \mathcal{F}$ such that $\mathbb{X} \subset \bigcup_{i=1}^\infty E_i$ and $d(E_i)\leq \delta$ for all $i$.

\item For every $\delta>0$, there exists $E\in\mathcal{F}$ with $d(E)\leq \delta$ such that $h(d(E))\leq \delta$.
\end{enumerate}

For any $0<\delta\leq\infty$ and $E\subset\mathbb{X}$, we define
\[
\Lambda^h_{\delta}(E)
:=
\inf\left\{
\sum_{i=1}^\infty h(d(E_i)):\,
E\subset\bigcup_{i=1}^\infty E_i,\,
d(E_i)<\delta,\,
E_i\in\mathcal{F}
\right\}.
\]

Assumption (1) ensures that such coverings always exist, while assumption (2) guarantees that $\Lambda^h_{\delta}(\emptyset)=0$. 
It is straightforward to verify that $\Lambda^h_{\delta}$ is monotone and subadditive, hence defines a measure. 
Moreover, for all $0<\delta_1<\delta_2\leq\infty$,
\[
\Lambda^h_{\delta_2}(E)\leq\Lambda^h_{\delta_1}(E).
\]

Hence, we define
\begin{equation}\label{generalsetup}
\Lambda^h(E)
:=
\lim_{\delta\to0} \Lambda^h_{\delta}(E)
=
\sup_{\delta>0} \Lambda^h_{\delta}(E),
\qquad \forall E\subset\mathbb{X}.
\end{equation}

As shown in \cite{MR1333890}, the measure $\Lambda^h$ is a Borel measure and is Borel regular if $\mathcal{F}\subset \mathcal{B}_{\mathbb{X}}$.

\subsubsection{General $s$-Hausdorff measure}

Let $\mathbb{X}$ be separable and let $0\leq s<\infty$. 
Taking $h(x)=x^s$, the resulting measure $\Lambda^h$ defined in \eqref{generalsetup} is called the (general) $s$-dimensional Hausdorff measure, denoted by $\Hau^s_{\mathbb{X}}$. Thus,
\[
\Hau^s_{\mathbb{X}}(E)
=
\lim_{\delta\to0} \Hau^s_{\mathbb{X},\delta}(E)
=
\sup_{\delta>0} \Hau^s_{\mathbb{X},\delta}(E),
\qquad \forall E\subset\mathbb{X},
\]
where
\[
\Hau^s_{\mathbb{X},\delta}(E)
=
\inf\left\{
\sum_{i=1}^\infty d(E_i)^s:\,
E\subset\bigcup_{i=1}^\infty E_i,\,
d(E_i)<\delta,\,
E_i\in\mathcal{F}
\right\}.
\]

\begin{prop}\label{basicprop}
For $0\leq s<t<\infty$ and $E\subset\mathbb{X}$:

\begin{enumerate}
\item $\Hau^s_{\mathbb{X}}(E)<\infty$ implies $\Hau^t_{\mathbb{X}}(E)=0$.

\item $\Hau^t_{\mathbb{X}}(E)>0$ implies $\Hau^s_{\mathbb{X}}(E)=\infty$.
\end{enumerate}
\end{prop}

By Proposition \ref{basicprop}, we define the Hausdorff dimension of a set $E\subset\mathbb{X}$ by
\begin{align*}
\HD^{\mathbb{X}}(E)&
=
\sup\{s:\Hau^s_{\mathbb{X}}(E)>0\}
=
\sup\{s:\Hau^s_{\mathbb{X}}(E)=\infty\}\\
&=
\inf\{t:\Hau^t_{\mathbb{X}}(E)<\infty\}
=
\inf\{t:\Hau^t_{\mathbb{X}}(E)=0\}.
\end{align*}

\subsubsection{$s$-Hausdorff measure adapted to a specific dyadic structure}

Motivated by \cite{MR4966567}, we introduce a specific dyadic structure and the corresponding $s$-Hausdorff measure adapted to this structure.

Let $N\geq1$ be a fixed positive integer. Define
\[
D^*=D^*[N]=\bigcup_{j\in\mathbb{Z}}D^N_j,
\qquad
D^N_j:=\left\{x+[0,2^{-jN}) : x\in2^{-jN}\mathbb{Z}\right\}.
\]
and
\[
D^*_J=D^*_J[N]:=\bigcup_{j\geq J}D^N_j,
\qquad \forall J\in\mathbb{Z}.
\]
For each $Q\in D^N_j$, we define the \textbf{''length''} of $Q$ by \begin{align}\label{length}
    \ell_{D^*}(Q):=2^{-j}.
\end{align}

We denote by $\Hau^s_{D^*,\delta}$ and $\Hau^s_{D^*}=\lim_{\delta\to0}\Hau^s_{D^*,\delta}$ the corresponding Hausdorff measures adapted to $D^*$, while $\Hau^s_\delta$ and $\Hau^s=\lim_{\delta\to0}\Hau^s_\delta$ denote the standard Hausdorff measures.

We also denote by $\HD^{D^*}$ the Hausdorff dimension adapted to $D^*$, and by $\HD$ the standard Hausdorff dimension.

\subsection{Interaction between $\Hau^s$ and $\Hau^s_{D^*,\,\infty}$}

The aim of this section is to study the relationship between the standard Hausdorff measure $\Hau^s$ and the Hausdorff content adapted to $D^*$, $\Hau^s_{D^*,\,\infty}$. 
The higher-dimensional case ($d\geq2$) has been established in \cite{MR4966567}, and their argument adapts directly to the one-dimensional setting.

\begin{lem}\label{standardVSgeneral}
If $E\subset\R$ is compact and $0\leq s/N<\HD(E)$, then $\Hau^s_{D^*,\,\infty}(E)>0$.
\end{lem}
We introduce a restricted version of the (general) $s$-Hausdorff content $\Hau^s_{D^*_J,\,\infty}$, defined by
\[
\Hau^s_{D^*_J,\,\infty}(E)
:=
\inf\left\{
\sum_{i=1}^\infty \ell_{D^*}(Q_i)^s:\,
E\subset\bigcup_{i=1}^\infty Q_i,\,
Q_i\in D^*_J
\right\}.
\]

From the definition, it is immediate that
\begin{align}\label{trivial}
\Hau^s_{D^*,\,\infty}(E)\leq \Hau^s_{D^*_J,\,\infty}(E),\quad\forall\,J.
\end{align}

\begin{lem}\label{equivalent}
If $E\subset\R$ is contained in some element of $D^*_J$ for some $J$, then
\[
\Hau^s_{D^*,\,\infty}(E)=\Hau^s_{D^*_J,\,\infty}(E),
\qquad \forall s\geq0.
\]
\end{lem}
Next, we record a Frostman lemma adapted to $\Hau^s_{D^*,\,\infty}$. 
The proof follows the same general approach as in \cite{MR3617376}. 
For any measure $\nu$, we denote by $\|\nu\|$ its total variation norm.

\begin{lem}\label{Frostman2}
If $E\subset\R$ is compact and $s\geq0$, then there exists a Borel measure $\nu$ supported on $E$ such that $\|\nu\|\geq \Hau^s_{D^*,\,\infty}(E)$ and for all $Q\in  D^*$
\[
\nu(Q)\leq \ell_{D^*}(Q)^s.
\]
\end{lem}

\begin{lem}\label{controloferergy}
There exists a decreasing function $\mathbf{f}:(0,\infty)\to[1,\infty)$ such that the following holds: 
For each constant $L\geq0$ and exponents $t,s$ with $0<t<1$ and $s>t+N-1$, one has
\[
\sup\left\{
I_t(\mu):
\supp(\mu)\subset[0,1],\,
\|\mu\|\leq1,\,
\sup_{Q\in D^*}\frac{\mu(Q)}{\ell_{D^*}(Q)^s}\leq L
\right\}
\leq L\cdot \mathbf{f}(s-t-N+1).
\]
\end{lem}
\section{Smoothing inequality for the curve family}\label{sec:3}
This section contributes to the desired smoothing inequality. The major ingredient has been proved in the paper \cite{MR4776384} of Christ and Zhou. For more references along this line, see, for example, \cite{MR4295087,2022MC}.
 In the present work, we establish a Sobolev estimate for the multilinear averaging operator defined along the curve $(t,\gamma(t))$, utilizing the trilinear smoothing inequality recently developed by Christ and Zhou \cite{MR4776384} as the primary tool.

We first recall the curve family in \cite{MR4776384}. Let $I_0\subset\mathbb{R}$ be an open interval and $\vec{\gamma}: I_0\to\mathbb{R}^2$ be real-analytic with $\frac{d\vec{\gamma}}{dt}$ vanishing nowhere. Write $\vec{\gamma}(t)=(\gamma_1(t),\,\gamma_2(t))$  and define
$J(t) = \gamma'_1(t)-\gamma_2'(t).$ The smoothing inequality requires the curve $\vec{\gamma}$ satisfy the following three hypothesis:
\newline {\bf Hypothesis 1.}\ 
$\gamma'_1,\gamma'_2$ are linearly independent over $\mathbb{R}$. 
\newline {\bf Hypothesis 2.}\ 
There do not exist $a,b_1,b_2\in\mathbb{C}$ 
with $a\ne 0$ and at least one $b_j\ne 0$ such that $\sum_{j=1}^2 b_je^{a\gamma_j(t)}$
is constant in $I_0$.
\newline {\bf Hypothesis 3.}\ 
\begin{align*} \label{hypothesis1}
|J(t)| + |J'(t)| \neq 0 \ \  \forall\,t\in I_0.
\end{align*}
The first hypothesis is equivalent to the assumption that the range of $\vec{\gamma}$ is not contained
in any affine subspace of $\mathbb{R}^2$. We remark that \begin{align*}
    \left\{(t,\gamma(t)):\,\gamma\in\Theta\right\}
\end{align*} meets the above conditions and the three hypotheses.

Let $\ell\in\mathbb{N}$ and $\chi\in C^\infty(\mathbb{R})$ be a function supported on $[1/2,4]$ so that $\chi(t)=1$ for all $t\in[1,2]$. Define $\chi_{(\ell)}(t):=\chi(2^\ell t)$.

Treating the smoothing inequality in \cite[Theorem 2.4]{MR4776384} as the black box, we apply the bilinear localization technique as well as the complex interpolation to obtain the desired smoothing inequality.

\begin{thm}[Smoothing inequality]\label{smoothing}~\\
 Let $\gamma\in\Theta$. Then  there exists $\sigma_0>0$ and a constant $\kappa>0$  such that for any $\ell\in\N$
   \begin{align*}
   \left\|\int_{\mathbb{R}} 
f_1(x-t)f_2(x-\gamma(t))\chi_{(\ell)}(t)\,dt\right\|_{H^{\sigma_0}} \lesssim 2^{\kappa\ell}\cdot\prod^2_{j=1}\|f_j\|_{H^{-\sigma_0}},\end{align*}
for all $f_1,f_2\in\mathcal{S}(\mathbb{R})$.
\end{thm}
\begin{proof}
Let $f_1,f_2,f_3\in\mathcal{S}(\mathbb{R})$, we define the bilinear operator
\[T(f_1,f_2)(x)=\int_{\mathbb{R}} 
f_1(x-t)f_2(x-\gamma(t))\chi_{(\ell)}(t)\,dt,\]
and its associated trilinear form
\[\mathcal{T}(f_1,f_2,f_3)=\int_{\R}T(f_1,f_2)(x)f_3(x)\;dx.\]
By a standard duality argument and the fact that $\mathcal{S}(\mathbb{R})$ is dense in $H^{-\sigma_0}(\mathbb{R})$, it suffices to show that 
\begin{align}\label{desiredddd}
\mathcal{T}(f_1,f_2,f_3)\lesssim2^{\kappa\ell}\cdot\prod^3_{j=1}\|f_j\|_{H^{-\sigma_0}}.\end{align}

For any $\eta\in C_0^{\infty}(\R)$, define the localized trilinear form $    \mathcal{T}_{\eta}$ associated with $\eta$ by
\begin{align*}
     \mathcal{T}_{\eta}(f_1,f_2,f_3):=\int_{\R}T(f_1,f_2)(x)f_3(x)\eta(x)\,dx,\quad\forall f_j\in\mathcal{S}(\mathbb{R}).
\end{align*}
We claim the following local estimate:
\begin{equation}\label{claim}
    \mathcal{T}_{\eta}(f_1,f_2,f_3)\lesssim2^{\kappa\ell}\cdot\prod^3_{j=1}\|f_j\|_{H^{-\sigma_0}}.
\end{equation}
We first assume the claim \eqref{claim} and verify (\ref{desiredddd}). Consider a smooth partition of unity of the real line with functions $\{\rho_n\}_{n\in\Z}$ supported on intervals of length $2$. Then there exists a finite constant $a$ \footnote{the existence of $a$ is guaranteed by the behavior $|\gamma(t)|\lesssim |t|$ is when $t$ close to 0.} depending on $\gamma$ such that the supports of $T(f_1\rho_{n_1}, f_2\rho_{n_2})$ and $f_3\rho_{n_3}$ intersect only if $|n_3 - n_1| \le 2$ and $|n_3 - n_2| \le a$. Consequently, we get
\begin{align}\label{partition}
    \mathcal{T}(f_1,f_2,f_3)\lesssim\sum_{\substack{|n_3-n_1|\leq2\\|n_3-n_2|\leq a}}\int_{\R}T(f_1\rho_{n_1},f_2\rho_{n_2})(x)(f_3\rho_{n_3})(x)\,dx.
\end{align}
Thus, using \eqref{claim}, we obtain 
\begin{align*}
    \mathcal{T}(f_1,f_2,f_3)&\lesssim\sum_{\substack{|n_3-n_1|\leq2\\|n_3-n_2|\leq a}}2^{\kappa\ell}\cdot\prod^3_{j=1}\|f_j\rho_{n_j}\|_{H^{-\sigma_0}}\\
    &\lesssim 2^{\kappa\ell}\|f_3\|_{H^{-\sigma_0}}\left(\sum_{n_1\in\Z}\|f_1\rho_{n_1}\|_{H^{-\sigma_0}}^2\right)^{\frac{1}{2}}\left(\sum_{n_2\in\Z}\|f_2\rho_{n_2}\|_{H^{-\sigma_0}}^2\right)^{\frac{1}{2}}\\
    &\lesssim2^{\kappa\ell}\cdot\prod^3_{j=1}\|f_j\|_{H^{-\sigma_0}},
\end{align*}
which is the desired inequality.

Now, it remains to verify the claim \eqref{claim}.\\ We first establish $L^p-$estimates for $T$. Applying the Cauchy-Schwarz inequality, we obtain
\begin{align}\label{C-Sforbilinear}
|T(f_1,f_2)(x)| \leq\left(\int_{\mathbb{R}} |f_1(x-t)|^2\chi_{(\ell)}(t)\,dt\right)^{\frac{1}{2}} \left(\int_{\mathbb{R}} |f_2(x-\gamma(t))|^2\chi_{(\ell)}(t)\,dt\right)^{\frac{1}{2}}.\end{align}
Observe that the first factor can be dominated pointwise by either $\|f_1\|_{L^2}$ or $2^{-\frac{\ell}{2}}(\mathcal{M}(|f_1|^2)(x))^{\frac{1}{2}}$, where $\mathcal{M}$ stands for the Hardy--Littlewood maximal function.\\
For the second factor, we note that the support of the smooth cutoff function satisfies $\operatorname{supp}(\chi_{(\ell)}) \subset [2^{-\ell-1}, 2^{-\ell+2}] =: I_\ell$ and $\gamma'$ does not vanish on $I_{\ell}$, then the change of variables $s = \gamma(t)$ yields 
\begin{align}\label{f2estimate}  
\left(\int_{\mathbb{R}} |f_2(x-\gamma(t))|^2\chi_{(\ell)}(t)\,dt\right)^{\frac{1}{2}} &\leq\left(\int_{\gamma(I_\ell)} |f_2(x-s)|^2\,\frac{ds}{|\gamma'(\gamma^{-1}(s))|}\right)^{\frac{1}{2}}\notag\\
&\leq\left(\frac{1}{c_\ell}\int_{-\gamma(2^{-\ell-1})}^{\gamma(2^{-\ell+2})} 
|f_2(x-s)|^2\,ds\right)^{\frac{1}{2}},\end{align}
where $c_\ell=\min\limits_{t\in I_\ell}|\gamma'(t)|>0$. Besides, since $\gamma$ is real-analytic, there exists $\kappa_0>0$ depending on $\gamma$ so that 
$$2^{-\kappa_0\ell} \simeq \min\limits_{t\in I_\ell}|\gamma'(t)|$$ and, together with (\ref{f2estimate}), we can deduce the pointwise bound 
\[\left(\int_{\mathbb{R}} |f_2(x-\gamma(t))|^2\chi_{(\ell)}(t)\,dt\right)^{\frac{1}{2}}\lesssim\min\left\{2^{-\frac{\ell}{2}}\left(\mathcal{M}(|f_2|^2)(x)\right)^{\frac{1}{2}},2^{\kappa_0\ell}\|f_2\|_{L^2}\right\}.\]
 Combining the bounds for both factors with (\ref{C-Sforbilinear}), we obtain
\begin{align}\label{L2L2}
    \|T(f_1,f_2)\|_{L^{\infty}}\lesssim2^{\kappa_0\ell}\|f_1\|_{L^{2}}\|f_2\|_{L^{2}},
\end{align}
and
\begin{align}\label{T-Holderestimates}
    T(f_1,f_2)(x)\lesssim2^{-\ell}\left(\mathcal{M}(|f_1|^2)(x)\right)^{\frac{1}{2}}\left(\mathcal{M}(|f_2|^2)(x)\right)^{\frac{1}{2}}.
\end{align}

Next, applying the Fubini's theorem and the change of variables, we have
\begin{align*}
    \int_{\R}|T(f_1,f_2)(x)|\;dx&\leq\int_{\R}\int_{\R}|f_1(x-t)f_2(x-\gamma(t))\chi_{(\ell)}(t)|\;dtdx\\
    &=\int_{\R}|f_1(x)|\int_{\R}|f_2(x+t-\gamma(t))\chi_{(\ell)}(t)|\;dxdt.
\end{align*}
 We analyze the behavior of the phase derivative $1-\gamma'(t)$ on the support $I_\ell = [2^{-\ell-1}, 2^{-\ell+2}]$ by considering two cases. 

\subsubsection*{Case $1$: The phase derivative has a root in $I_\ell$.} Without loss of generality, assume the root occurs at $2^{-\ell-1}$ so that $\gamma'(2^{-\ell-1})=1$. By the hypotheses on $\gamma$, we have $\gamma''(2^{-\ell-1}) \neq 0$ and hence the Taylor's theorem implies that $|1-\gamma'(t)|$ is comparable to $|2^{-\ell-1}-t|$ on $I_{\ell}$.

We isolate the singularity by performing a dyadic decomposition of the operator around the root:
$$T(f_1,f_2)(x) = \sum_{k\geq0} \int_{I_{\ell,k}} f_1(x-t)f_2(x-\gamma(t))\chi_{(\ell)}(t)\;dt =: \sum_{k\geq0} T_k(f_1,f_2)(x),$$
where $I_{\ell,k} = [2^{-\ell-1}(1+7 \cdot 2^{-k-1}), 2^{-\ell-1}(1+7 \cdot 2^{-k})]$.\\
Due to the real analyticity of $\gamma$, we deduce that $|1-\gamma'(t)| \simeq 2^{-k-(\kappa_0+1)\ell}$ for all $t \in I_{\ell,k}$ and $\kappa_0 \geq 0$ depending on $\gamma$, which gives that
\begin{align}\label{L^111forTk} 
\|T_k(f_1,f_2)\|_{L^1}\lesssim2^{k+(\kappa_0+1)\ell}\|f_1\|_{L^1}\|f_2\|_{L^1}.\end{align}
To sum these localized estimates, from the Cauchy--Schwarz inequality and (\ref{L^111forTk}), we tile $\mathbb{R} = \cup_{n} I_n$ with disjoint intervals of length $2^{-\ell-k}$ and obtain
\begin{align*}
    \|T_k(f_1,f_2)\|_{L^{1/2}}^{\frac{1}{2}}&\leq\sum_{\substack{|n-n_1|\leq2\\|n-n_2|\leq a}}\int_{I_n}|T(f_1\chi_{I_{n_1}},f_2\chi_{I_{n_2}})(x)|^{\frac{1}{2}}\;dx\\
    &\leq \sum_{\substack{|n-n_1|\leq2\\|n-n_2|\leq a}}2^{-\frac{\ell+k}{2}}\left(\int_{I_n}|T(f_1\chi_{I_{n_1}},f_2\chi_{I_{n_2}})(x)|\;dx\right)^{\frac{1}{2}}\\
    &\lesssim\sum_{\substack{|n-n_1|\leq2\\|n-n_2|\leq a}}2^{-\frac{\ell+k}{2}}2^{\frac{k+(\kappa_0+1)\ell}{2}}\|f_1\chi_{I_{n_1}}\|_{L^1}^{\frac{1}{2}}\|f_2\chi_{I_{n_2}}\|_{L^1}^{\frac{1}{2}}\\
    &\lesssim2^{\frac{\kappa_0\ell}{2}}\|f_1\|_{L^1}^{\frac{1}{2}}\|f_2\|_{L^1}^{\frac{1}{2}},
\end{align*}
where we use the fact that for a fixed $n$ there are finitely many $n_1,n_2$ for which $T_k(f_1\chi_{I_{n_1}},f_2\chi_{I_{n_2}})$ is non-zero. 

Interpolating the above $L^1\times L^1\to L^{1/2}-$estimate with the trivial bound $$\|T_k(f_1,f_2)\|_{L^{\infty}} \leq 2^{-k-\ell}\|f_1\|_{L^{\infty}}\|f_2\|_{L^{\infty}}$$ and summing in $k\geq0$, we obtain
\begin{equation}\label{Holder1}
\|T(f_1,f_2)\|_{L^{p/2}} \lesssim 2^{\kappa_0\ell}\|f_1\|_{L^p}\|f_2\|_{L^p}, \quad \forall p>1.
\end{equation}

\subsubsection*{Case $2$: The phase derivative is non-vanishing.} If $1-\gamma'(t)$ is non-vanishing on $I_\ell$, a further change of variables $s = x+t-\gamma(t)$ immediately implies 
$$\|T(f_1,f_2)\|_{L^1} \lesssim \|f_1\|_{L^1}\|f_2\|_{L^1}.$$
Following a localization argument similar to the previous case, along with Cauchy-Schwarz inequality and interpolation with trivial $L^{\infty}\times L^{\infty}\to L^{\infty}-$bound, we can get
$$\|T(f_1,f_2)\|_{L^{p/2}} \lesssim \|f_1\|_{L^p}\|f_2\|_{L^p}$$
for all $p \geq 1$.

We recall the estimate for trilinear form $\mathcal{T}_{\eta}$ from \cite[Theorem 2.4]{MR4776384} which guarantees the existence of $p < \infty$ and $\tau > 0$ such that
\begin{equation}\label{Christ}
    \mathcal{T}_{\eta}(f_1,f_2,f_3)\lesssim2^{\kappa_1\ell}\prod_{j=1}^3\|f_j\|_{W^{p,-\tau}}.
\end{equation}
Note that if $p=2$ in the equation above, it gives us \eqref{claim} with $\sigma_0=\tau$. It remains to prove \eqref{claim} when $p \neq 2$ in \eqref{Christ}.\\
Consider the case $p<2$ in \eqref{Christ}. From \eqref{T-Holderestimates}, we get $T$ maps $L^3\times L^3\to L^{\frac{3}{2}}$, which implies
\begin{equation}\label{L3}
    \mathcal{T}_{\eta}(f_1,f_2,f_3)\lesssim2^{-\ell}\prod_{j=1}^3\|f_j\|_{L^{3}}.
\end{equation}
Now, applying the complex interpolation \cite[Theorem 7.2.9]{Grafakos14} between \eqref{Christ} and \eqref{L3} proves the desired claim in \eqref{claim} for suitable constants $\kappa$ and $\sigma_0>0$.

Next, we consider $p>2$ in \eqref{Christ}. Interpolating between \eqref{Holder1} with $p_1=p_2=\frac{3}{2}$ and \eqref{L2L2}, we obtain that $T$ maps $L^{9/5}(\R)\times L^{9/5}(\R)$ to $L^{9/4}(\R)$ with operator norm bounded by $O(2^{\kappa_0\ell})$. Consequently, we get
\begin{equation}\label{L9/5}
    \mathcal{T}_{\eta}(f_1,f_2,f_3)\lesssim2^{\kappa_0\ell}\prod_{j=1}^3\|f_j\|_{L^{9/5}}.
\end{equation}
Again, using the complex interpolation between \eqref{Christ} and \eqref{L9/5} yields the desired claim \eqref{claim}.
\end{proof}

\begin{remark}
    We remark that in the paper of M. Hsu and F. Y.-H. Lin \cite{FredMartin}, a general smoothing inequality is established. As indicated in their preprint, page $4$, this smoothing inequality guarantees the existence of the pattern  \begin{align*}
        \left\{(x,y),\,(x-t,y),\,(x,y+\Gamma(t))\right\}
    \end{align*}
     for a set $E\subset\mathbb{R}^2$ with positive Lebesgue measure, where $\Gamma$ is the curve family. We highlight that, in particular, this result guarantees the existence of the two-point pattern in the real line whenever the set has positive Lebesgue measure.
\end{remark}

\subsection{Consequence of the smoothing inequality}
Let $\phi\in\mathcal{S}(\mathbb{R})$ be a non-negative and radial smooth bump function such that $\widehat{\phi}(0)=1$, $\phi(x)\geq\frac{1}{2}$ for all $|x|\leq\frac{1}{2}$ and $\supp(\phi)\subset [-1,1]$. For any $\varepsilon>0$, define $\phi_\varepsilon(x):=\varepsilon^{-1}\phi(x/\varepsilon)$ to be the $L^1$-dilation of $\phi$ for all $x\in\mathbb{R}$.  For any $\varepsilon>0$, define $\mu_\varepsilon:=\mu*\phi_\varepsilon$. 
\begin{cor}\label{smoothingprop}  Let $\gamma\in\Theta$. There exists $\sigma_0>0$ and $\kappa>0$ such that for all $\sigma\in(0,\sigma_0)$ 
    \begin{align*}
        \left|\iint f_1(x-t)f_2(x-\gamma(t))\chi_{(\ell)}(t)\,dt\,d\mu(x)\right|\lesssim 2^{\kappa\ell}\cdot\prod^2_{j=1}\|f_j\|_{H^{-\sigma}}\cdot\|\mu\|_{H^{-\sigma}},
    \end{align*}
    for all $f_1,f_2\in\mathcal{S}(\mathbb{R}).$
\end{cor}

\begin{proof}
     Since $\mu_\varepsilon$ converges weakly to $\mu$ as $\varepsilon\to 0^+$, then
     \begin{align*}
    &\left|\iint f_1(x-t)f_2(x-\gamma(t))\chi_{(\ell)}(t)\,dt\,d\mu(x)\right|\\
    &=\lim_{\varepsilon\to 0^+}     \left|\iint \mu_\varepsilon(x)f_1(x-t)f_2(x-\gamma(t))\chi_{(\ell)}(t)\,dt\,dx\right|,    
     \end{align*}
     and hence, by the duality argument and Theorem \ref{smoothing}, 
       \begin{align*}
    \left|\iint f_1(x-t)f_2(x-\gamma(t))\chi_{(\ell)}(t)\,dt\,d\mu(x)\right|&\lesssim\lim_{\varepsilon\to 0^+}       2^{\kappa\ell}\cdot\prod^2_{j=1}\|f_j\|_{H^{-\sigma}}\cdot\|\mu_\varepsilon\|_{H^{-\sigma}}\\
    &=2^{\kappa\ell}\cdot\prod^2_{j=1}\|f_j\|_{H^{-\sigma}}\cdot\|\mu\|_{H^{-\sigma}}.
     \end{align*}
\end{proof}
With the Corollary \ref{smoothingprop}, we can prove the configuration proposition which gives a path to show the existence of a certain pattern.
\begin{prop}\label{PropofPATTERN}
Given $\gamma\in\Theta$. Let $\sigma_0>0$ the parameters given in Corollary \ref{smoothingprop} and let $\mu\in\mathbf{M}(E)$ with the condition that $I_{1-\sigma}(\mu)<\infty$ for some $\sigma\in(0,\sigma_0)$.  Suppose that 
\begin{align}\label{patterncondition}
\liminf_{\varepsilon\to0}\iint \mu_\varepsilon(x-t)\mu_\varepsilon(x-\gamma(t))\chi_{(\ell)}(t)\,dt\,d\mu(x)>0,
\end{align}
for some $\ell\in\mathbb{N}\cup\{0\}$.    Then there is $x\in E$ and non-zero $t\in\mathbb{R}$ such that $x,x-t,x-\gamma(t)\in E.$
\end{prop}

The following technical lemma is needed before we show this proposition. Fix $\gamma\in\Theta$, define a linear function $ \vartheta$ on functions $f:[0,1]^2\to\mathbb{R}$ by
\begin{align}\label{functional}
  \vartheta(f):=\liminf_{\varepsilon\to0}\iint f(x,y)\mu_\varepsilon(y)\mu_\varepsilon(x-\gamma(x-y))\chi_{(\ell_0)}(x-y)\,dy\,d\mu(x).
\end{align}
\begin{lem}\label{Lemmafunction}
    The limit in (\ref{functional}) exists for all continuous function $f$ and
    \begin{align*}
    \left|\vartheta(f)\right|\leq C\cdot\|f\|_{L^\infty},
\end{align*}
for some $C>0$, independent of $f$.
\end{lem}
\begin{proof}
For every $\varepsilon>0$, we have
\begin{align*}
 &\left|\iint f(x,y)\mu_\varepsilon(y)\mu_\varepsilon(x-\gamma(x-y))\chi_{(\ell_0)}(x-y)\,dy\,d\mu(x) \right|  \\
 &\leq\|f\|_{L^\infty}\cdot\iint \mu_\varepsilon(y)\mu_\varepsilon(x-\gamma(x-y))\chi_{(\ell_0)}(x-y)\,dy\,d\mu(x),
\end{align*}
which is bounded by, from Corollary \ref{smoothingprop},
\begin{align*}
2^{\kappa\ell_0}\cdot\|f\|_{L^\infty}\cdot\|\mu_\varepsilon\|^2_{H^{-\sigma}}\cdot\|\mu\|_{H^{-\sigma}}\lesssim 2^{\kappa\ell_0}\cdot\|f\|_{L^\infty}\cdot\|\mu\|^3_{H^{-\sigma}},\quad\forall \sigma\in(0,\sigma_0).
\end{align*}
By the assumption and Proposition \ref{EnergyandSobolev}, there exists $C$, independent of $f$ such that
\begin{align*}
    \left|\vartheta(f)\right|\leq C\cdot\|f\|_{L^\infty},
\end{align*}
if the limit (\ref{functional}) in exists.

To verify the existence of the limit, by the Stone--Weierstrass theorem, it remains to show that the limit exists for every smooth function $f$ whose Fourier series consists of only finitely many terms. The existence of the limit then follows from the Cauchy criterion and Theorem \ref{smoothing}.
\end{proof}

Now, we are ready to prove Proposition \ref{PropofPATTERN}.
\begin{proof}
By Lemma \ref{Lemmafunction} and the Riesz representation theorem, we get a non-negative measure $\vartheta$ defined by (\ref{functional}) and $\vartheta([0,1]^2)>0$ from the hypothesis (\ref{patterncondition}).

 Define the closed set
\begin{align*}
\mathcal{P}:=\{(x, y)\in [0, 1]^2: x, y, x-\gamma(x-y)\in E\}.
\end{align*}
By the definition of the measure $\vartheta$, it suffices to prove that $\vartheta$ is supported on $\mathcal{P}$.\\
Let $f$ be a continuous function with $\supp(f)$ disjoint from $\mathcal{P}$. We need to prove that $\vartheta( f)=0$. Note that $\text{dist}(\supp(f), \mathcal{P})>0$. Using a partition of unity, we are able to write $f$ as a finite sum $\sum f_j$, where for each $j$, the function $f_j$ is continuous and satisfies at least one of the following:
\begin{align}\label{firstsecond}
\begin{cases}
     \text{dist}(\supp(f_j), E\times [0, 1])>0,\\
 \text{dist}(\supp(f_j), [0, 1]\times E\})>0,
\end{cases}
\end{align}
or
\begin{align}\label{third}
\text{dist}\Big(\big\{x-\gamma(x-y): (x, y)\in \supp(f_j)\big\}, E\Big)>0.
\end{align}
We will prove that $\langle \nu, f_j\rangle=0$ for every $j$. If $f_j$ satisfies either the first or the second condition in (\ref{firstsecond}), then the integral in (\ref{functional}) vanishes for every $\epsilon$ small enough. If $f_j$ satisfies the third condition in (\ref{third}), then the support of $f_j$ is a positive distance from the support of $\mu_{\epsilon}(x -\gamma(x-y))$ for sufficiently small $\epsilon$, so the integral (\ref{functional}) is again vanishes if $\epsilon$ is sufficiently small. The proof is complete.
\end{proof}

\section{Construction of measure with energy and spectral gap conditions}\label{sec:4}
In this section, we will construct a measure that satisfies {\it the energy} and {\it the spectral gap condition}. This construction was first introduced in \cite{MR4609784}. For the generalization of the construction of the measure in the one dimension and higher dimension case, readers are recommended to refer to the paper \cite{MR3617376,MR5038694}.
We modify those ideas in our construction. 
\begin{lem}\label{lemmaofLoca}
    Let $E\subset[0,1]$ be a compact subset with large Hausdorff dimension $\HD(E)>1-\varepsilon$ and $N\geq1$. Define $s=s(N):=N(1-\varepsilon)$. Then for each $\delta>0$ and $J$, there is $Q\in D^*_J$ such that
    \begin{align}\label{Locate}
        \Hau^s_{D^*_J,\,\infty}(E\cap Q)\geq (1-\delta) \ell_{D^*}(Q)^s;
    \end{align}
    equivalently, for each $\delta>0$ and $J$, there is $Q\in D^*_J$ such that
    \begin{align}\label{Locate'}
        \Hau^s_{D^*,\,\infty}(E\cap Q)\geq (1-\delta) \ell_{D^*}(Q)^s.
    \end{align}
\end{lem}
\begin{proof}
  From Lemma \ref{standardVSgeneral} and (\ref{trivial}), $0< \Hau^s_{D^*,\,\infty}(E)\leq \Hau^s_{D^*_J,\,\infty}(E)$ for all $J$. Trivially, from the definition of the $s$-Hausdorff content and the compactness of $E$, $\Hau^s_{D^*_J,\,\infty}(E)<\infty$.
  
  We prove (\ref{Locate}) by getting a contradiction. Suppose not, then there is $\delta_0>0$ and $J_0$ so that for all $ Q\in D^*_{J_0}$ \begin{align}\label{pcontra}
 \Hau^s_{D^*_{J_0},\,\infty}(E\cap Q)< (1-\delta_0) \ell_{D^*}(Q)^s, \end{align}
    and remark that  for each sufficiently small $\varepsilon'>0$, there is a cover of $E$, $\{Q_j\}^\infty_{j=1}\subset D^*_{J_0}$, such that
    \begin{align*}
        \sum^\infty_{j=1}  \ell_{D^*}(Q_j)^s\leq \varepsilon'+\Hau^s_{D^*_{J_0},\,\infty}(E).
    \end{align*}
   Hence, by the subadditivity of the (general) $s$-dimensional Hausdorff content and (\ref{pcontra}),  \begin{align*}
0<\Hau^s_{D^*_{J_0},\,\infty}(E)&\leq\sum^\infty_{j=1} \Hau^s_{D^*_{J_0},\,\infty}(E\cap Q_j)\\&<(1-\delta_0)\sum^\infty_{j=1}  \ell_{D^*}(Q_j)^s\leq  (1-\delta_0)\left(\varepsilon'+\Hau^s_{D^*_{J_0},\,\infty}(E)\right
   )<\infty, \end{align*}
   which is a contradiction. Thus, we complete the proof of (\ref{Locate}).    
   
   By Lemma \ref{equivalent}, the equivalence of (\ref{Locate'}) is established and the proof is done.
\end{proof}

We need one more notation for the construction of the desired measure.\\ Let $Q\in D^N_j$ and write $Q=L_Q+(0,2^{-jN}]$, where $L_Q$ is the left-endpoint of $Q$. Define the rescaling operator $\mathbf{T}_Q:Q\to [0,1]$ associated with $Q$ by
\begin{align}\label{def:rescaling}
    \mathbf{T}_Q(\cdot):=2^{jN}(\cdot-L_Q).
\end{align}

\begin{prop}\label{constructESG}Let $N\geq1$ be an integer  and $1<A<B$. Let $0<\mathbf{t}<1$ and $J\geq0$ be given, then there exists $0<\varepsilon_N=\varepsilon_N(A,B,\mathbf{t})<1$ such that the following property holds: For each compact subset $E\subset[0,1]$ with large Hausdorff dimension $\HD(E)>1-\varepsilon_N$ and  there exists a probability measure $\mu$ supported on the closure of $\mathbf{T}_\mathbf{Q}(E\cap\mathbf{Q})$  for some $\mathbf{Q}\in D^*_J$ so that the energy condition
    \begin{align}\label{energy+spectral1}
       I_{\mathbf{t}}(\mu)\lesssim1
    \end{align}
    and the spectral gap condition
    \begin{align}\label{energy+spectral2}
        \int_{|\xi|\in [A^{1/4},B^2]}\left|\widehat{\mu}(\xi)\right|\,d\xi\lesssim A^{-4}.
    \end{align}
    are satisfied.
\end{prop}

\begin{proof}
Let $T=T(B)\gg1$ be a sufficiently large constant depends on $B$ and we define \begin{align}\label{choiceofparameter}
 0<\varepsilon_N:=\min\left\{\frac{\log_2(1+2^{-NT-2})}{NT},\,\frac{1}{4N^2},\,\frac{1-\mathbf{t}}{2N}\right\}\leq\frac{1}{2}.
\end{align}
To establish the spectral gap condition, fix $\varphi:\mathbb{R}\to\mathbb{R}$ be a non-negative, smooth bump function supported on $[0,1]$ such that $\widehat{\varphi}(0)=1$ and  $\|\varphi\|_{L^\infty}=2$. Then
for all $M\geq1$,\begin{align*}
    \int_{|\xi|\geq A^{1/4}}|\widehat{\varphi}(\xi)|\,d\xi\leq C_M\int_{|\xi|\geq A^{1/4}}|\xi|^{-M}\,d\xi= C_M A^{\frac{1-M}{4}}.
\end{align*}
\smallskip

Let $E\subset[0,1]$ be a compact subset with large Hausdorff dimension $\HD(E)>1-\varepsilon_N$ and let $s=s(N):=N(1-\varepsilon_N)$. Then, from Lemma \ref{lemmaofLoca}, for $\delta=2^{-NT-2}$, there is $\mathbf{Q}\in D^*_J$ so that 
\begin{align}\label{firstchoice}
        \Hau^s_{D^*,\,\infty}(E\cap \mathbf{Q})\geq (1-\delta)\ell_{D^*}(\mathbf{Q})^s.
    \end{align}

Let $ch_T(\mathbf{Q})$ be the family of $T^{th}$-children $\mathbf{q}$ of $\mathbf{Q}$ with respect to the dyadic system $D^*=D^*[N]$; in other words,
$$ch_T(\mathbf{Q})=\left\{\mathbf{q}\in D^*:\,\mathbf{q}\subset\mathbf{Q},\,\ell_{D^*}(\mathbf{q})=2^{-T}\ell_{D^*}(\mathbf{Q})\right\}$$We claim that for each element $\mathbf{q}\in ch_T(\mathbf{Q})$, there holds that
\begin{align}\label{lower}
     \Hau^s_{D^*,\,\infty}(E\cap \mathbf{q})\geq \frac{1}{2}\ell_{D^*}(\mathbf{q})^s.
\end{align}
To verify this, let $\mathbf{G}$ be the set of all elements $\mathbf{q}$ in $ch_T(\mathbf{Q})$ such that (\ref{lower}) holds. Suppose not, then $\#\left(ch_T(\mathbf{Q})\setminus{\mathbf G}\right)\geq1$, and hence, by the subadditivity of the (general) $s$-dimensional Hausdorff content, \begin{align*}
 \Hau^s_{D^*,\,\infty}(E)&\leq\sum_{\mathbf{q}\in \mathbf{G}}  \Hau^s_{D^*,\,\infty}(E\cap\mathbf{q})+ \sum_{\mathbf{q}\in ch_T(\mathbf{Q})\setminus\mathbf{G}}  \frac{1}{2}\ell_{D^*}(\mathbf{q})^s\\
 &\leq\sum_{\mathbf{q}\in \mathbf{G}}  \ell_{D^*}(\mathbf{q})^s+ \sum_{\mathbf{q}\in ch_T(\mathbf{Q})\setminus\mathbf{G}}  \frac{1}{2}\ell_{D^*}(\mathbf{q})^s\\
 &=\sum_{\mathbf{q}\in ch_T(\mathbf{Q})}  \ell_{D^*}(\mathbf{q})^s-\sum_{\mathbf{q}\in ch_T(\mathbf{Q})\setminus\mathbf{G}}  \frac{1}{2}\ell_{D^*}(\mathbf{q})^s,
 \end{align*}
 which leads to, from the property  (\ref{firstchoice}) and the choice of $s=s(N)$,  \begin{align*}
     1-2^{-NT-2}\leq \frac{\Hau^s_{D^*}(E)}{\ell(\mathbf{Q})^s}&\leq \sum_{\mathbf{q}\in ch_T(\mathbf{Q})}  \frac{\ell_{D^*}(\mathbf{q})^s}{\ell_{D^*}(\mathbf{Q})^s}-\sum_{\mathbf{q}\in ch_T(\mathbf{Q})\setminus\mathbf{G}}  \left(\frac{1}{2}\cdot\frac{\ell_{D^*}(\mathbf{q})^s}{\ell_{D^*}(\mathbf{Q})^s}\right)\\
     &\leq 2^{NT-Ts}-2^{-Ts-1}\\
     &<2^{\varepsilon_N NT}-2^{-NT-1}\\
       &\leq 1+ 2^{-NT-2}-2^{-NT-1},
 \end{align*}
 which contradicts to the choice of $\varepsilon_N$ that $\varepsilon_N\leq\log_2(1+2^{-NT-2})/{NT}$. Hence, the claim (\ref{lower}) is obtained.

 Therefore, from Lemma \ref{Frostman2} and the claim (\ref{lower}), for each $\mathbf{q}\in ch_T(\mathbf{Q})$, there is a measure $\widetilde{\nu_\mathbf{q}}$ supported on $E\cap \mathbf{q}$ such that\begin{align}\label{lowerbound}
     \|\widetilde{\nu_\mathbf{q}}\|\geq\Hau^s_{D^*,\,\infty}(E\cap \mathbf{q})\geq\frac{1}{2}\ell_{D^*}(\mathbf{q})^s\quad\text{and}\quad \widetilde{\nu_\mathbf{q}}(Q)\leq \ell_{D^*}(Q)^s,\quad\forall\,Q\in D^*.
 \end{align}
For each $\mathbf{q}\in ch_T(\mathbf{Q})$, define the measure $\nu_\mathbf{q}$, which is supported on $\mathbf{q}\cap E$, by
 \begin{align*}
    \nu_\mathbf{q}(\cdot):=\frac{\varphi(\mathbf{q})\cdot\ell_{D^*}(\mathbf{Q})^s}{\widetilde{\nu_\mathbf{q}}(\mathbf{q})}\cdot\widetilde{\nu_\mathbf{q}}(\cdot),\quad\text{where}\quad\varphi(\mathbf{q}):=\int_{\mathbf{T}_\mathbf{Q}(\mathbf{q})}\varphi(x)\,dx
 \end{align*}
 and define the measure $\nu$ to be\begin{align*}
     \nu:=\sum_{\mathbf{q}\in ch_T(\mathbf{Q})}  \nu_\mathbf{q}.
 \end{align*}
 By the construction of the measure $\widetilde{\nu_\mathbf{q}}$, it is easy to see that the measure $\nu$ is supported on $E\cap\overline{\mathbf{Q}}$ and the total variation of $\nu$ is
 \begin{align*}
 \sum_{\mathbf{q}\in ch_T(\mathbf{Q})}  \|\nu_\mathbf{q}\|=\sum_{\mathbf{q}\in ch_T(\mathbf{Q})}  \varphi(\mathbf{q})  \cdot\ell_{D^*}(\mathbf{Q})^s=\widehat{\varphi}(0)\cdot \ell_{D^*}(\mathbf{Q})^s=   \ell_{D^*}(\mathbf{Q})^s;
 \end{align*}
 besides, it can be shown that the measure $\nu$ is well-defined, that is $\nu(\mathbf{Q})>0$.
 
 Now, we show that the measure $\nu$ obeys the Frostman condition on $D^*$, that is 
\begin{align}\label{ballfornu}
    \nu(Q)\leq 4\ell_{D^*}(Q)^s,\quad\forall\,Q\in D^*.
\end{align}
Let $Q\in D^*$ be given. If $\ell_{D^*}(Q)>\ell_{D^*}(\mathbf{Q})$, then $\nu(Q)$ is bounded by the total variation of $\nu$, which leads to that 
 \begin{align*}
\nu(Q)\leq\ell_{D^*}(\mathbf{Q})^s\leq\ell_{D^*}(Q)^s.
 \end{align*}
 To analyze the case $\ell_{D^*}(Q)\leq\ell_{D^*}(\mathbf{Q})$. We remark that, from the trivial estimate of $\varphi$, the property (\ref{def:rescaling}) and the choice of $s=s(N)$, we have \begin{align*}
\varphi(\mathbf{q})\cdot\ell_{D^*}(\mathbf{Q})^s\leq \|\varphi\|_{L^\infty}|\mathbf{T}_\mathbf{Q}(\mathbf{q})|\cdot\ell_{D^*}(\mathbf{Q})^s\leq 2\left(\frac{\ell_{D^*}(\mathbf{q})}{\ell_{D^*}(\mathbf{Q})}\right)^N\cdot\ell_{D^*}(\mathbf{Q})^s
\leq 2\ell_{D^*}(\mathbf{q})^s,
\end{align*}
and hence combines with (\ref{lowerbound}), the lower bound of the total variation of the measure $\widetilde{\nu_\mathbf{q}}$, implies that
\begin{align}\label{smallprop}
\varphi(\mathbf{q})\cdot\ell_{D^*}(\mathbf{Q})^s\leq 4\|\widetilde{\nu_\mathbf{q}}\|.
\end{align}
Hence, if $\ell_{D^*}(Q)\leq 2^{-T}\ell_{D^*}(\mathbf{Q})$, then there is at most one $\mathbf{q}\in ch_T(\mathbf{Q})$ such that $Q\subset\mathbf{q}$ and hence (\ref{smallprop}) and (\ref{lowerbound}), the Frostman condition of $\widetilde{\nu_\mathbf{q}}$, lead to
 \begin{align*}
\nu(Q)\leq\nu_{\mathbf{q}}(Q)=\frac{\varphi(\mathbf{q})\cdot\ell_{D^*}(\mathbf{Q})^s}{\widetilde{\nu_\mathbf{q}}(\mathbf{q})}\cdot\widetilde{\nu_\mathbf{q}}(Q)\leq 4\ell_{D^*}(Q)^s; 
 \end{align*}
while, if $2^{-T}\ell_{D^*}(\mathbf{Q})<\ell_{D^*}(Q)\leq \ell_{D^*}(\mathbf{Q})$, then, from the construction of $\nu_\mathbf{q}$,
 \begin{align*}
     \nu(Q)&\leq\#\{\mathbf{q}\in ch_T(\mathbf{Q}):\,q\cap Q\neq\emptyset\}\cdot\max_{\mathbf{q}\in ch_T(\mathbf{Q})}\|\nu_\mathbf{q}\|\\
     &\leq\left(\frac{\ell_{D^*}(Q)}{2^{-T}\ell_{D^*}(\mathbf{Q})}\right)^N\cdot\max_{\mathbf{q}\in ch_T(\mathbf{Q})}\left\{2\left(\frac{\ell_{D^*}(\mathbf{q})}{\ell_{D^*}(\mathbf{Q})}\right)^N\cdot\ell_{D^*}(\mathbf{Q})^s\right\}\\
     &\leq2\ell_{D^*}(Q)^s,
 \end{align*}
 therefore, we get the desired Frostman condition for $\nu$ on $D^*$, that is (\ref{ballfornu}).\\

 Now, we construct the desired measure $\mu$ by defining that
 \begin{align*}
     \mu(K):=\frac{\nu(\mathbf{T}^{-1}_\mathbf{Q}(K)\cap\mathbf{Q})}{\nu(\mathbf{Q})},\quad\forall\,K\subset\mathbb{R}:\text{Borel~measurable};
 \end{align*}
 it is easy to see that  $\mu$ is a probability measure and the support of $\mu$ is contained in the closure of $\mathbf{T}_\mathbf{Q}(\supp(\nu)\cap\mathbf{Q})$, which is a subset of $[0,1]$.\\
 We first show that the Frostman condition holds for $\mu$. For each $Q\in D^*$, from the construction of $\mu$ and the value of the total variation of $\nu$ and property (\ref{ballfornu}), we have
 \begin{align}\label{ballmu}
     \mu(Q):=\frac{\nu(\mathbf{T}^{-1}_\mathbf{Q}(Q)\cap\mathbf{Q})}{\nu(\mathbf{Q})}=\frac{\nu(\mathbf{T}^{-1}_\mathbf{Q}(Q))}{\|\nu\|}\leq\frac{4\ell_{D^*}(\mathbf{T}^{-1}_\mathbf{Q}(Q))^s}{ \ell_{D^*}(\mathbf{Q})^{s}}=4\ell_{D^*}(Q)^s,
 \end{align}
which gives the Frostman condition for $\mu$.\\ To obtain the energy condition (\ref{energy+spectral1}), we apply Lemma \ref{controloferergy} with the suitable choice of parameters. Let $L=4$, then Lemma \ref{controloferergy}, together with (\ref{ballmu}) and (\ref{choiceofparameter}), guarantee that \begin{align*}
    I_{\mathbf{t}}(\mu)\leq L\cdot\mathbf{f}(1-\mathbf{t}-\varepsilon_NN)<\infty,
\end{align*}
which verifies (\ref{energy+spectral1}) due to the requirement on $\varepsilon_N$ that $\varepsilon_N\leq (1-\mathbf{t})/{2N}$\\
Next, we check the spectral gap condition (\ref{energy+spectral2}) for $\mu$. For each $\mathbf{q}\in ch_T(\mathbf{Q})$, one has that
  \begin{align*}
\mu(\mathbf{T}_\mathbf{Q}(\mathbf{q}))=\frac{\nu_{\mathbf{q}}(\mathbf{q})}{\nu(\mathbf{Q})}=\frac{\varphi(\mathbf{q})\cdot\ell_{D^*}(\mathbf{Q})^s}{\nu(\mathbf{Q})}=\varphi(\mathbf{q}),
  \end{align*}
  which implies that
  \begin{align*}
      |\widehat{\mu}(\xi)-\widehat{\varphi}(\xi)|&\leq \sum_{\mathbf{q}\in ch_T(\mathbf{Q})}\left|\int_{\mathbf{T}_\mathbf{Q}(\mathbf{q})}e^{-2\pi i x\xi}\,d\mu(x)-\int_{\mathbf{T}_\mathbf{Q}(\mathbf{q})}e^{-2\pi i x\xi}\,d\varphi(x)\right|\\
      &\leq \sum_{\mathbf{q}\in ch_T(\mathbf{Q})}\int_{\mathbf{T}_\mathbf{Q}(\mathbf{q})}\left|e^{-2\pi i x\xi}-e^{-2\pi i c_{\mathbf{T}_\mathbf{Q}(\mathbf{q})}\xi}\right|\,d\mu(x)\\
      &+\sum_{\mathbf{q}\in ch_T(\mathbf{Q})}\int_{\mathbf{T}_\mathbf{Q}(\mathbf{q})}\left|e^{-2\pi i x\xi}-e^{-2\pi i c_{\mathbf{T}_\mathbf{Q}(\mathbf{q})}\xi}\right|\,d\varphi(x).
  \end{align*}
  Hence, if we apply the mean value theorem to the phase function, then
    \begin{align*}
      |\widehat{\mu}(\xi)-\widehat{\varphi}(\xi)|\lesssim 2^{-T}|\xi|.
  \end{align*}
To finish the proof, let $M,T\gg1$ , then there is $C_M>0$ such that
 \begin{align*}
   \int_{|\xi|\in [A^{1/4},B^2]}\left|\widehat{\mu}(\xi)\right|\,d\xi&\leq   \int_{|\xi|\in [A^{1/4},B^2]}\left|\widehat{\mu}(\xi)-\widehat{\varphi}(\xi)\right|\,d\xi+\int_{|\xi|\geq [A^{1/4}}\left|\widehat{\varphi}(\xi)\right|\,d\xi\\
   &\leq\int_{|\xi|\in [A^{1/4},B^2]}\left|\widehat{\mu}(\xi)-\widehat{\varphi}(\xi)\right|\,d\xi+ C_M A^{\frac{1-M}{4}}\\
   &\leq C\cdot2^{-T}B^4+ C_M A^{\frac{1-M}{4}}\\
   &\lesssim A^{-4},
 \end{align*}
 which verifies (\ref{energy+spectral2}) and the proof is complete.
    \end{proof}
\begin{remark}
    Compared to the S\'ark\"ozy-type problem, we remark that the parameter $N$ plays no role in proving the existence of the nonlinear pattern on the real line. In fact, the verification of Proposition \ref{constructESG} for $N=1$ is strong enough for us to show the main theorem. However, we record the Proposition \ref{constructESG} for general $N$ for future use.
\end{remark}

\section{Reduction scheme: the lower bound of the configuration integral and the proof of theorem \ref{mainthm}}\label{sec:5}
\subsection{Reduction scheme}\label{sec:5.1}
In this section, we decompose the configuration integral and show that for the measure that satisfies the energy and the spectral gap conditions, it satisfies the configuration integral condition (\ref{patterncondition}).

\begin{prop}\label{Lemmaofconfcheck}
Let $\gamma\in\Theta(1)$ be any given curve and $N\geq1$.
Let $\sigma_0>0$ and  $\kappa>0$ be the parameters given by the Corollary \ref{smoothingprop} and let $\sigma\in(0,\sigma_0)$. Then there exists constants $1<A<B$ such that for any probability measure  $\mu$  given by Proposition \ref{constructESG} with parameter $\mathbf{t}\in(1-\sigma,1)$, 
\begin{align}\label{configurationintegraldec}
\liminf_{\varepsilon\to0}\iint \mu_\varepsilon(x-t)\mu_\varepsilon(x-\gamma(t))\chi_{(\ell)}(t)\,dt\,d\mu(x)\gtrsim A^{-1},
   \end{align}
  for some sufficiently large $\ell\geq\ell_0$. 
\end{prop}

\subsubsection{Decomposition of the configuration integral}
Through this section, we fix a curve $\gamma\in\Theta(1)$.
We use high-low frequency analysis to decompose the configuration integral. For any $1<A<B$, we write
\begin{align}\label{decofmeasure}
\mu_\varepsilon=\mu_{1/A}+\left(\mu_{1/B}-\mu_{1/A}\right)+\left(\mu_{\varepsilon}-\mu_{1/B}\right),
\end{align}
then the configuration integral
\begin{align*}
\iint \mu_\varepsilon(x-t)\mu_\varepsilon(x-\gamma(t))\chi_{(\ell)}(t)\,dt\,d\mu(x)
   \end{align*}
can be decomposed into the sum of the main term 
\begin{align}\label{mainterm}
     \iint \mu_{1/A}(x-t)\mu_{1/A}(x-\gamma(t))\chi_{(\ell)}(t)\,dt\,d\mu(x)+,
\end{align}
and the error terms $(\rm I)+(\rm II)+(\rm III)+(\rm IV)$,
where 
\begin{align}\label{Errorterm1}
    (\rm I)&:= \iint\mu_{1/A}(x-t)\left(\mu_{1/B}-\mu_{1/A}\right)(x-\gamma(t))\chi_{(\ell)}(t)\,dt\,d\mu(x)\\
&\quad+\iint\mu_{1/A}(x-t)\left(\mu_{\varepsilon}-\mu_{1/B}\right)(x-\gamma(t))\chi_{(\ell)}(t)\,dt\,d\mu(x).\notag\\
&=(\rm I)_1+(\rm I)_2.\notag  
\end{align}
and
\begin{align}\label{Errorterm2}
     (\rm II)&:= \iint\left(\mu_{1/B}-\mu_{1/A}\right)(x-t)\left(\mu_{1/B}-\mu_{1/A}\right)(x-\gamma(t))\chi_{(\ell)}(t)\,dt\,d\mu(x)\\
&\quad+\iint\left(\mu_{\varepsilon}-\mu_{1/B}\right)(x-t)\left(\mu_{1/B}-\mu_{1/A}\right)(x-\gamma(t))\chi_{(\ell)}(t)\,dt\,d\mu(x),\notag\\
&=(\rm II)_1+(\rm II)_2.\notag
\end{align}
and
\begin{align}\label{Errorterm3}
     (\rm III)&:= \iint\left(\mu_{1/B}-\mu_{1/A}\right)(x-t)\left(\mu_{\varepsilon}-\mu_{1/B}\right)(x-\gamma(t))\chi_{(\ell)}(t)\,dt\,d\mu(x)\\
&\quad+\iint\left(\mu_{\varepsilon}-\mu_{1/B}\right)(x-t)\left(\mu_{\varepsilon}-\mu_{1/B}\right)(x-\gamma(t))\chi_{(\ell)}(t)\,dt\,d\mu(x),\notag\\
&=(\rm III)_1+(\rm III)_2.\notag
\end{align}
and
\begin{align}\label{Errorterm4}
   (\rm IV)&:=\iint \left(\mu_{1/B}-\mu_{1/A}\right)(x-t)\mu_{1/A}(x-\gamma(t))\chi_{(\ell)}(t)\,dt\,d\mu(x)\\
     &\quad+ \iint \left(\mu_{\varepsilon}-\mu_{1/B}\right)(x-t)\mu_{1/A}(x-\gamma(t))\chi_{(\ell)}(t)\,dt\,d\mu(x),\notag\\
     &=(\rm IV)_1+(\rm IV)_2.\notag
\end{align}

\subsubsection{Estimate on the main term}
In this section, we deal with the quantitative estimate of the main term.
\begin{lem}\label{LEMMAMAINTERM}
    Let $A=A(\ell)$ be a given large positive number given by $(4A)^{-1}=2^{1-\ell} $with the property that $ |\gamma(t)|\leq |t|$ for all $|t|\leq \frac{1}{4A}$. Then the main term (\ref{mainterm}) is bounded below by $\frac{1}{409600A}$.
\end{lem}
\begin{proof}
    For each $c>0$, define 
    \begin{align*}
 \mathbf{D}_c:=\left\{x\in\mathbb{R}:\exists\,0<r_x\leq1\,s.t.\,\mu(B_{r_x}(x))\leq c r_x\right\},       
    \end{align*}
    then the set $\mathbf{D}_c$ can be cover by the union of $B_{r_x}(x)$ over all $x\in \mathbf{D}_c$. By the Vitali covering lemma \cite[Chapter 1]{MR0867284}, there exists a sequence of radius $\{r_j\}_j\subset\mathbb{R}_{+}$ and a family of points $\{x_j\}_j\subset \mathbf{D}_c$ such that
    \begin{align*}
\mathbf{D}_c\subset\bigcup^\infty_{j=1}B_{r_j}(x_j)\quad\text{and}\quad \left\{B_{r_j/5}(x_j)\right\}_j:\text{pairwise~disjoint}.        
    \end{align*}
Besides, it is easy to see that the union of $B_{r_j/5}(x_j)$ is contained in $[-2,2]$. Then, by the subadditivity of the measure, 
    \begin{align*}
        \mu(\mathbf{D}_c)\leq \sum^\infty_{j=1}\mu(B_{r_j}(x_j))\leq \sum^\infty_{j=1}cr_j=\frac{5c}{2}\sum^\infty_{j=1}\left|B_{r_j/5}(x_j)\right|\leq10c,\quad\forall\,c>0, 
    \end{align*}
    which implies that there is $c_0=\frac{1}{20}$ so that $\mu(\mathbf{D}_{c_0})\leq \frac{1}{2}$.

Let $x\notin \mathbf{D}_{c_0}$, then  $\mu(B_{r}(x))>c_0 r$ for all $0<r\leq1$ and hence for all $|t|\leq\frac{1}{4A}$,
\begin{align*}
    \mu_{1/A}(x-t)&=A\int\phi(A(x-t-y))\,d\mu(y)\\
    &\geq \frac{A}{2}\cdot \mu\left(B_{\frac{1}{2A}} (x-t)\right)\\
      &\geq \frac{A}{2}\cdot \mu\left(B_{\frac{1}{4A}} (x)\right)\\
      &>\frac{c_0}{8},
\end{align*}
in which the first inequality follows from the property of $\phi$ that $\phi(x)\geq\frac{1}{2}$ for all $|x|\leq\frac{1}{2}$. Similarly, from the  growing condition on $\gamma$ and the choice of $A$, one has
\begin{align*}
    \mu_{1/A}(x-\gamma(t))
  >\frac{c_0}{8},\quad\forall\,|t|\leq\frac{1}{4A},\,\,x\notin \mathbf{D}_{c_0}.
\end{align*}
Therefore, by the construction of $\chi$, 
\begin{align*}
   & \iint \mu_{1/A}(x-t)\mu_{1/A}(x-\gamma(t))\chi_{(\ell)}(t)\,dt\,d\mu(x)\\
   &\geq   \int_{(\mathbf{D}_{c_0})^c}\int_{|t|\leq\frac{1}{4A}} \mu_{1/A}(x-t)\mu_{1/A}(x-\gamma(t))\chi_{(\ell)}(t)\,dt\,d\mu(x)\\
    &\geq  \frac{c^2_0}{64} \cdot\mu((\mathbf{D}_{c_0})^c)\int_{|t|\leq\frac{1}{4A}} \chi_{(\ell)}(t)\,dt;
\end{align*}
and thus, by the choice of $A$, one has
\begin{align*}
   &\iint \mu_{1/A}(x-t)\mu_{1/A}(x-\gamma(t))\chi_{(\ell)}(t)\,dt\,d\mu(x)\\
    &\geq \frac{c^2_0}{128}\cdot \left|[0,\frac{1}{4A}]\cap[2^{-\ell},\,2^{-\ell+1}]\right|\\
    &\geq\frac{1}{409600A}, 
\end{align*} 
which completes the proof.
\end{proof}
\subsubsection{Estimate on the minor terms}
Before we start to estimate the minor terms, we need one more proposition, which indicates that the Sobolev norm of the high frequency part, $\mu_\varepsilon-\mu_{1/B}$, is negligible.\\
Recall that $\phi$ is a radial function, hence $\frac{\partial\widehat{\phi}}{\partial t}(0)=0$ and therefore for each $\xi$, there is $\xi'\in\mathbb{R}$ with $|\xi'|\leq|\xi|$ such that
    \begin{align}\label{meanvalue}
        \left|\widehat{\phi}(\xi)-\widehat{\phi}(0)\right|=|\xi|\left|\frac{\partial\widehat{\phi}}{\partial t}(\xi')-\frac{\partial\widehat{\phi}}{\partial t}(0)\right| \leq C_{\phi}\cdot|\xi|^2.
    \end{align}
   
\begin{prop}\label{88888}
 Let  $0<\varepsilon<1$ be given and $B>0$. For each $\sigma\in(0,1)$, if $1-\sigma<\mathbf{t}<1$, then
 \begin{align*}
     \|\mu_\varepsilon-\mu_{1/B}\|^2_{H^{-\sigma}}\lesssim \varepsilon^4B+B^{-3}+B^{\frac{1-\sigma-\mathbf{t}}{5}}I_\mathbf{t}(\mu),
 \end{align*}
 where the implicit constant depends on the function $\phi$.
\end{prop}
\begin{proof}
    From (\ref{Sobolevdef}), we have   \begin{align*}
        \|\mu_{\varepsilon}-\mu_{1/B}\|^2_{H^{-\sigma}}&:=\int|\widehat{\mu}(\xi)|^2\cdot  \left|\widehat{\phi}(\varepsilon\xi)-\widehat{\phi}(\xi/B)\right|^2 \left(1+|\xi|^2\right)^{-\sigma/2}\,d\xi\\
        &= \int_{|\xi|\leq B^{1/5}}+\int_{|\xi|> B^{1/5}}.
    \end{align*}
    By the triangle inequality and the estimate (\ref{meanvalue}), the first integral can be dominated by
    \begin{align*}
  \int_{|\xi|\leq B^{1/5}}  \left|\widehat{\phi}(\varepsilon\xi)-\widehat{\phi}(\xi/B)\right|^2\,d\xi=&\int_{|\xi|\leq B^{1/5}}  \left|\widehat{\phi}(\varepsilon\xi)-\widehat{\phi}(0)+\widehat{\phi}(0)-\widehat{\phi}(\xi/B)\right|^2\,d\xi\\  &\leq C_{\phi}\cdot   \int_{|\xi|\leq B^{1/5}}  \left|\varepsilon^2+B^{-2}\right|^2\cdot|\xi|^4\,d\xi \\
   &=C_{\phi}\cdot\left|\varepsilon^2+B^{-2}\right|^2\cdot B;
    \end{align*}
  for the second integral, by the trivial estimate $\|\widehat{\phi}\|_{L^\infty}\lesssim1$ and Proposition \ref{EnergyandSobolev}, it is bounded by
    \begin{align*}
        4\int_{|\xi|> B^{1/5}}|\widehat{\mu}(\xi)|^2\left(1+|\xi|^2\right)^{-\sigma/2}\,d\xi &\leq 4 B^{\frac{1-\sigma-\mathbf{t}}{5}}\int_{|\xi|> B^{1/5}}|\widehat{\mu}(\xi)|^2|\xi|^{\mathbf{t}-1}\,d\xi\\
        &\lesssim B^{\frac{1-\sigma-\mathbf{t}}{5}}I_\mathbf{t}(\mu),
    \end{align*}
provided that $1-\sigma<\mathbf{t}<1$. The proof is complete.
\end{proof}
For the estimates of the minor terms, we will utilize the spectral-gap property of the measure $\mu$ to get an upper bound for each term.
\begin{lem}\label{SPECTRALESTIMATE}
   Let $N\geq1$ and given any constants $1<A<B$. Let $\mu$ be a probability measure that satisfies the spectral gap condition (\ref{energy+spectral2}). Then if $A$ is sufficiently large, the following estimates hold:
    \begin{align}\label{spectlemma1}
        \int\left|\widehat{\mu}(\xi)\right|\left|\widehat{\phi}(\xi/B)-\widehat{\phi}(\xi/A)\right|\,d\xi\lesssim A^{-5/4}
    \end{align}
    and for each $\sigma>0$
\begin{align}\label{spectlemma2}
        \|\mu_{1/B}-\mu_{1/A}\|^2_{H^{-\sigma}}\lesssim A^{-11/4},\end{align}where the implicit constants in (\ref{spectlemma1}) and (\ref{spectlemma2}) depends only on the function $\phi$.
\end{lem}

\begin{proof}
    To apply the spectral gap property of the measure $\mu$ to estimate (\ref{spectlemma1}), we decomposition the integral into
    \begin{align*}
         \int\left|\widehat{\mu}(\xi)\right|\left|\widehat{\phi}(\xi/B)-\widehat{\phi}(\xi/A)\right|\,d\xi= \int_{|\xi|\leq A^{1/4}}+\int_{|\xi|\in[A^{1/4},B^2]}+\int_{|\xi|>B^2}.
    \end{align*}
    By the triangle inequality and spectral gap property (\ref{energy+spectral2}), the second integral is dominated by
    \begin{align*}
2\int_{|\xi|\in[A^{1/4},B^2]}\left|\widehat{\mu}(\xi)\right|\,d\xi\leq 2 A^{-4};
    \end{align*}
    for the first and the third integral, since the $L^\infty$-norm of $\widehat{\mu}$ is bounded by $1$, then
    \begin{align}\label{estimate1}
       &\int_{|\xi|\leq A^{1/4}}\left|\widehat{\mu}(\xi)\right|\left|\widehat{\phi}(\xi/B)-\widehat{\phi}(\xi/A)\right|\,d\xi+\int_{|\xi|>B^2} \left|\widehat{\mu}(\xi)\right|\left|\widehat{\phi}(\xi/B)-\widehat{\phi}(\xi/A)\right|\,d\xi\notag\\
       &\leq \int_{|\xi|\leq A^{1/4}}\left|\widehat{\phi}(\xi/B)-\widehat{\phi}(\xi/A)\right|\,d\xi+\int_{|\xi|>B^2} \left|\widehat{\phi}(\xi/B)-\widehat{\phi}(\xi/A)\right|\,d\xi\notag\\
      &\leq \int_{|\xi|\leq A^{1/4}}\left|\widehat{\phi}(\xi/B)-\widehat{\phi}(\xi/A)\right|\,d\xi+\int_{|\xi|>B^2} \left|\widehat{\phi}(\xi/B)\right|\,d\xi+\int_{|\xi|>B^2} \left|\widehat{\phi}(\xi/A)\right|\,d\xi .
    \end{align}
   Hence, by applying (\ref{meanvalue}) to the former in (\ref{estimate1}) and the Schwartz decay estimate, $|\widehat{\phi}(\xi)|\leq C_{\phi}|\xi|^{-5}$, to the latter in (\ref{estimate1}), it is further bounded by
    \begin{align*}
 A^{-2} \int_{|\xi|\leq A^{1/4}}\left|\xi\right|^2\,d\xi+B^{-3}\lesssim A^{-5/4},
    \end{align*}
     which completes (\ref{spectlemma1}).

    Next, we estimate (\ref{spectlemma2}). Since $\sigma>0$, then 
    \begin{align*}
        \|\mu_{1/B}-\mu_{1/A}\|^2_{H^{-\sigma}}&:=\int|\widehat{\mu}(\xi)|^2\cdot  \left|\widehat{\phi}(\xi/B)-\widehat{\phi}(\xi/A)\right|^2 \left(1+|\xi|^2\right)^{-\sigma/2}\,d\xi\\
       &\leq \int|\widehat{\mu}(\xi)|^2\cdot  \left|\widehat{\phi}(\xi/B)-\widehat{\phi}(\xi/A)\right|^2 \,d\xi.
    \end{align*}
 By the same decomposition, we can deduce that $ \|\mu_{1/B}-\mu_{1/A}\|^2_{H^{-\sigma}}$ is dominated by $A^{-11/4}$.  
\end{proof}

Now, we are ready to estimate the minor terms. We first demonstrate the estimates on the terms $(\rm I)_1$ and $(\rm I)_2$.
\begin{lem}\label{ERI}
    Let $N\geq1$ and given any constants $1<A<B$, let $\mu$ be a probability measure satisfies the spectral gap condition (\ref{energy+spectral2}), then
     \begin{align*}
          | (\rm I)_1|\lesssim 2^{-\ell} A^{-1/4},
     \end{align*}
     where the implicit constant depends on the $L^1$-norm of $\chi$ and depends on the function $\phi$.
\end{lem}
\begin{proof}
By the Fourier representation and the change of variable, we have
    \begin{align*}
          (\rm I)_1&= \iint\mu_{1/A}(x-t)\left(\mu_{1/B}-\mu_{1/A}\right)(x-\gamma(t))\chi_{(\ell)}(t)\,dt\,d\mu(x)\\
&=2^{-\ell}\iint\widehat{\mu}(\xi+\eta)\overline{\widehat{\mu_{1/A}}}(\eta)\left(\overline{\widehat{\mu_{1/B}}}-\overline{\widehat{\mu_{1/A}}}\right)(\xi)\int e^{-2\pi i(t\eta+\gamma(2^{-\ell}t)\xi)}\chi(t)\,dt\,d\xi\,d\eta,
    \end{align*}
    and hence, from the fact that $\|\widehat{\mu}\|_{L^\infty}\leq1$, one has
    \begin{align}\label{nocurveinvolve}
      | (\rm I)_1|&\leq   2^{-\ell}\|\chi\|_{L^1}\iint\left|\overline{\widehat{\mu_{1/A}}}(\eta)\right|\left|\left(\overline{\widehat{\mu_{1/B}}}-\overline{\widehat{\mu_{1/A}}}\right)(\xi)\right|\,d\xi\,d\eta\notag\\
      &\simeq 2^{-\ell}\int\left|\widehat{\mu}(\eta)\right|\left|\widehat{\phi}(\eta/A)\right|\,d\eta\cdot\int\left|\widehat{\mu}(\xi)\right|\left|\widehat{\phi}(\xi/B)-\widehat{\phi}(\xi/A)\right|\,d\xi.\notag\\
      &\leq 2^{-\ell}\int\left|\widehat{\phi}(\eta/A)\right|\,d\eta\cdot\int\left|\widehat{\mu}(\xi)\right|\left|\widehat{\phi}(\xi/B)-\widehat{\phi}(\xi/A)\right|\,d\xi.
    \end{align}
    To get the upper bound of (\ref{nocurveinvolve}), since that $\widehat{\phi}$ is a Schwartz function on $\mathbb{R}$, then, by the change of variable, the $L^1$-norm of $\widehat{\phi}(\cdot/A)$ is dominated by $A$; combines with (\ref{spectlemma1}) in Lemma \ref{SPECTRALESTIMATE} and (\ref{nocurveinvolve}), the proof is complete.
\end{proof}

\begin{lem}\label{ERII}
   Let $N\geq1$ and given any constants $1<A<B$. Let $\sigma_0>0$ and  $\kappa>0$ be the parameters given by the Corollary \ref{smoothingprop} and let $\sigma\in(0,\sigma_0)$. Then for each probability measure $\mu$ satisfies the spectral gap condition (\ref{energy+spectral2}), 
     \begin{align*}
          | (\rm I)_2|\lesssim 2^{\kappa\ell} \cdot I_{1-\sigma}(\mu)\cdot \left(\varepsilon^4B+B^{-3}+B^{\frac{1-\mathbf{t}-\sigma}{5}}I_\mathbf{t}(\mu)\right)^{1/2},
     \end{align*}
     for all $\ell\in\N$ and $1-\sigma<\mathbf{t}<1$ and the implicit constant depends on the function $\phi$.
     \end{lem}
\begin{proof}
    By the Corollary \ref{smoothingprop}, we have
    \begin{align*}
     \left|(\rm I)_2\right|\lesssim2^{\kappa\ell}\cdot\|\mu_{1/A}\|_{H^{-\sigma}}\cdot\|\mu_\varepsilon-\mu_{1/B}\|_{H^{-\sigma}}\cdot\|\mu\|_{H^{-\sigma}}.
\end{align*}
From Proposition \ref{EnergyandSobolev} and Proposition \ref{88888}, we further have that
    \begin{align*}
     \left|(\rm I)_2\right|&\lesssim_{\gamma}2^{\kappa\ell}\cdot I_{1-\sigma}(\mu)\cdot\|\mu_\varepsilon-\mu_{1/B}\|_{H^{-\sigma}}\\
     &\lesssim2^{\kappa\ell}\cdot I_{1-\sigma}(\mu)\cdot \left(\varepsilon^4B+B^{-3}+B^{\frac{1-\mathbf{t}-\sigma}{5}}I_\mathbf{t}(\mu)\right)^{1/2}, \end{align*}
    for all $1-\sigma<\mathbf{t}<1$.
\end{proof}

The rest of this section contributes to the estimate of the remaining minor terms. These estimates are similar to those in Lemmas \ref{ERI} and \ref{ERII}.

\begin{lem}\label{ERIII}
    Let $N\geq1$ and given any constants $1<A<B$. Then for each probability measure $\mu$ satisfies the spectral gap condition (\ref{energy+spectral2}), there holds 
     \begin{align*}
         \left|({\rm II})_1\right|\lesssim 2^{-\ell} A^{-5/2}\quad\text{and}\quad \left|({\rm IV})_1\right|\leq 2^{-\ell} A^{-1/4},
     \end{align*}
    where the implicit constant depends on the $L^1$-norm of $\chi$ and depends on the function $\phi$ for all sufficiently large $\ell>0$.
     Besides, let $\sigma_0>0$ and  $\kappa>0$ be the parameters given by the Corollary \ref{smoothingprop} and let $\sigma\in(0,\sigma_0)$, then for all sufficiently large $\ell$ and
     $1-\sigma<\mathbf{t}<1$\begin{align*}
          | (\rm II)_2|,\,| (\rm III)_1|\lesssim 2^{\kappa\ell}\cdot A^{-11/8}\cdot I^{1/2}_{1-\sigma}(\mu)\cdot \left(\varepsilon^4B+B^{-3}+B^{\frac{1-\mathbf{t}-\sigma}{5}}I_\mathbf{t}(\mu)\right)^{1/2} 
     \end{align*}
     and 
     \begin{align*}
| (\rm III)_2|\lesssim 2^{\kappa\ell}\cdot I^{1/2}_{1-\sigma}(\mu)\cdot \left(\varepsilon^4B+B^{-3}+B^{\frac{1-\mathbf{t}-\sigma}{5}}I_\mathbf{t}(\mu)\right),
     \end{align*}
     and
      \begin{align*}
| (\rm IV)_2|\lesssim 2^{\kappa\ell}\cdot I_{1-\sigma}(\mu)\cdot \left(\varepsilon^4B+B^{-3}+B^{\frac{1-\mathbf{t}-\sigma}{5}}I_\mathbf{t}(\mu)\right)^{1/2},
     \end{align*}
     where the implicit constants depend on $\phi$.
\end{lem}
\begin{proof}
     We sketch the proof. The idea is based on the estimates in Lemma \ref{ERI} and Lemma \ref{ERII}. For the terms $({\rm II})_1, ({\rm IV})_1$, we apply the Fourier representation method; while, for the other terms, we apply the smoothing inequality and the Proposition \ref{88888}.
\end{proof}

\subsection{Proof of Proposition \ref{Lemmaofconfcheck} }
 Let $\ell\gg100$ be sufficiently large and fix $A=2^{\ell-3}$ as in Lemma \ref{LEMMAMAINTERM}. Take $B\gg 2^{\frac{10(\kappa+1)\ell}{\mathbf{t}-(1-\sigma)}}$. Recall that the probability measure $\mu$ satisfies
\begin{align*}
     I_{\mathbf{t}}(\mu)\lesssim1\quad\text{and}\quad  \int_{|\xi|\in [A^{1/4},B^2]}\left|\widehat{\mu}(\xi)\right|\,d\xi\lesssim A^{-4},
\end{align*}
where $1-\sigma<\mathbf{t}<1$.

 By Lemma \ref{ERI}, Lemma \ref{ERII} and Lemma \ref{ERIII}, for all sufficiently large  $\ell$,
\begin{align*}
     &\left|({\rm I})\right|+ \left|({\rm II})\right|+ \left|({\rm III})\right|+ \left|({\rm IV})\right|\\&\lesssim 2^{-\ell} A^{-1/4}+2^{\kappa\ell} \cdot\left(\varepsilon^4B+B^{-3}+B^{\frac{1-\mathbf{t}-\sigma}{5}}\right)^{1/2}+2^{\kappa\ell}\cdot \left(\varepsilon^4B+B^{-3}+B^{\frac{1-\mathbf{t}-\sigma}{5}}\right).
\end{align*}
As a consequence, (\ref{configurationintegraldec}) is verified by taking $\liminf_{\varepsilon\to0}$, the suitable choice of $\ell$ and $B$, and Lemma \ref{LEMMAMAINTERM}. The proof is complete.

\subsection{Proof of  Theorem \ref{mainthm}} Let $\sigma_0>0$ and  $\kappa>0$ be the parameters given by the Corollary \ref{smoothingprop} and let $\sigma\in(0,\sigma_0)$. Pick a parameter $1-\sigma<\mathbf{t}<1$ and let $\varepsilon_N$ be given in Proposition \ref{constructESG}. If $E\subset[0,1]$ is a compact subset with large Hausdorff dimension $\HD(E)>1-\varepsilon_N$, then, from Proposition \ref{constructESG}, for any given $1<A<B$ there is a probability measure  a probability measure $\mu$ supported on the closure of $\mathbf{T}_\mathbf{Q}(E\cap\mathbf{Q})$  for some $\mathbf{Q}\in D^*_J$ satisfies \begin{align*}
        I_{\mathbf{t}}(\mu)\lesssim1\quad\text{and}\quad  \int_{|\xi|\in [A^{1/4},B^2]}\left|\widehat{\mu}(\xi)\right|\,d\xi\lesssim A^{-4}.
\end{align*}
By Proposition \ref{Lemmaofconfcheck}, we choose suitable $A<B$ such that for some sufficiently large $\ell>0$
\begin{align*}
    \liminf_{\varepsilon\to0}\iint \mu_\varepsilon(x-t)\mu_\varepsilon(x-\gamma(t))\chi_{(\ell)}(t)\,dt\,d\mu(x)\gtrsim A^{-1}.
\end{align*}
As a result, equipped with the energy condition of $\mu$ and the lower bound estimate of the configuration integral with the Proposition \ref{PropofPATTERN}, there is $y\in \mathbf{T}_\mathbf{Q}(E\cap\mathbf{Q})$ and non-zero $s\in\mathbb{R}$ so that
\begin{align*}
    \left\{y,\,y-s,\,y-\gamma(s)\right\}\subset \mathbf{T}_\mathbf{Q}(E\cap\mathbf{Q}).
\end{align*}
Henceforth, there is $j\geq J$ such that
\begin{align*}
    \left\{2^{-jN}y+L_{\mathbf{Q}},\,2^{-jN}(y-s)+L_{\mathbf{Q}},\,2^{-jN}(y-\gamma(s))+L_{\mathbf{Q}}\right\}\subset E\cap\mathbf{Q}\subset E,
\end{align*}
and the proof is fulfilled by taking 
\begin{align*}
    x=2^{-jN}y+L_{\mathbf{Q}},\quad t=2^{-jN}s,\quad \text{and}\quad \lambda=2^{jN}.
\end{align*}

\bigskip
\bigskip

\noindent {\bf Acknowledgment:}
The authors would like to thank Martin Hsu for introducing the smoothing inequality in \cite{FredMartin, studyguild}, and Fred Yu-Hsiang Lin for helpful discussions and for bringing the paper \cite{BK} to our attention.

The first author is supported by the National Center for Theoretical Sciences and the National Science and Technology Council of Taiwan under Grant No.~115-2124-M-002-009. The second and third authors are supported by the National Science and Technology Council (NSTC) under Grant No.~111-2115-M-002-010-MY5.

On behalf of all authors, the corresponding author declares that there is no conflict of interest.

\end{document}